\documentclass{amsart}
\usepackage{amssymb}
 \pdfoutput=1
\usepackage{graphics}
\usepackage{amsfonts}
\usepackage{graphicx}
\usepackage{epstopdf}
\usepackage[lite]{amsrefs}
\usepackage{amsthm}
\usepackage[usenames,dvipsnames,svgnames,table]{xcolor}
\usepackage{tikz}
\usepackage{url}
\usepackage{xifthen}

\theoremstyle{plain}
\newtheorem{theorem}{Theorem}[section]
\numberwithin{equation}{section}
\numberwithin{figure}{section}
\theoremstyle{plain}
\theoremstyle{remark}
\theoremstyle{plain}
\theoremstyle{definition}
\theoremstyle{conjecture}

\newtheorem{problem}[theorem]{Problem}
\newtheorem{lemma}[theorem]{Lemma}

\newtheorem{corollary}[theorem]{Corollary}

\newcommand{\cubed}[1]{#1 \times #1 \times #1}  

\newcommand{\mir}{^{*}}                           

\newcommand{\fr}{\mathrm {fr}  }

\newcommand{\flip}[2]
{\ifthenelse{#1<#2}{\{#1#2\}}{\{#2#1\}}}
\newcommand{\cubepic}[8]
{\draw (0+#7,0+#8) grid (1+#7,4+#8);
\draw (-1+#7,1+#8) grid (0+#7,2+#8);
\draw (1+#7,1+#8) grid (2+#7,2+#8);
\node at (0.5+#7,0.5+#8) {#1};
\node at (0.5+#7,1.5+#8) {#2};
\node at (0.5+#7,2.5+#8) {#3};
\node at (0.5+#7,3.5+#8) {#4};
\node at (-0.5+#7,1.5+#8) {#5};
\node at (1.5+#7,1.5+#8) {#6};
\node at (0.5+#7,-0.5+#8){\flip{#1}{#3}\flip{#2}{#4}\flip{#5}{#6}     };   
 }

\begin{document}

\title{Automorphisms of $S_6$ and the Colored Cubes Puzzle}
\author{Ethan Berkove, David Cervantes Nava, Daniel Condon, and Rachel Katz}
\address{ Mathematics Department, Lafayette College,  Quad Drive, Easton, PA 18042, USA }
\email{berkovee@lafayette.edu}
\urladdr{http://math.lafayette.edu/people/ethan-berkove/}
\address{State University of New York at Potsdam}
\email{cervand195@potsdam.edu}
\address{Georgia Institute of Technology}
\email{dcondon6@gatech.edu}
\address{University of Chicago}
\email{katz.rachelbrooke@gmail.com}
\thanks{This project was funded by NSF grant DMS-1063070.  We also gratefully acknowledge support from the Provost Office at Lafayette College.}
\date{\today}

\begin{abstract}
Given a palette of six colors, a \textit{colored cube} is a cube where each face is colored with exactly one color and each color appears on some face.  Starting with an arbitrary collection of unit length colored cubes, one can try to arrange a subset of the collection into a $\cubed{n}$ cube where each face is a single color.  This is the \textit{Colored Cubes Puzzle}.  In this paper, we determine minimum size sets of cubes required to complete an $\cubed{n}$ cube's \textit{frame}, its corners and edges.  We answer this problem for all $n$, and in particular show that for $n \geq4$ one has the best possible result, that as long as there are enough cubes to build a frame it can always be done, regardless of the cubes in the collection.  Part of our analysis involves the set of $6$-colored cubes and its associated $S_6$ action.  In addition to the problem simplification this action provides, it also gives another way to visualize the outer automorphism of $S_6$.  
\end{abstract}

\maketitle

\section{Introduction}  
Given a palette of six colors, a ``colored cube'' is one where each face of the cube is one color and all six colors appear on some face.    A pleasant combinatorial argument shows there are exactly $30$ distinct colored cubes, which can be used in a number of interesting puzzles.  For example: 

\begin{enumerate}
\item \label{Prob1} Select one colored cube.  Find seven other distinct colored cubes and assemble all eight into a $2 \times 2 \times 2$ cube where each face is the same color.
\item \label{Prob2} Proceed as above, but select the cubes so that all touching internal faces also have matching colors.  
\item\label{Prob3} Find $27$ distinct colored cubes that can be used to construct a $\cubed{3}$ cube where each face is the same color.
\end{enumerate}
The first two questions where addressed by Percy MacMahon nearly a century ago.  MacMahon (1854 -- 1929) is possibly best known for authoring one of the first books on enumerative combinatorics \cite{MacMahon01} and for serving as the President of the London Mathematical Society.  However, he also had a proclivity towards mathematical recreations, and addressed and answered the first two questions in \cite{MacMahon02}.  The third problem appears on the website \cite{Conway}.  

Our motivating paper \cite{BerkoveEtAl} investigated an extension of the third problem on the list, where one takes $n^3$ \textit{arbitrary} colored cubes and determines when it is possible to construct an $n \times n \times n$ cube where each face is the same color.  This is the formal ``Colored Cubes Puzzle.''  The following is the main theorem from \cite{BerkoveEtAl}.

\begin{theorem} \label{thm:priormain}
Let $n > 2$.  Given $n^3$ arbitrary colored cubes, it is always possible to solve the Colored Cubes Puzzle.
\end{theorem}

As noted in \cite{BerkoveEtAl}, there is a variation of this problem that is a better indicator of its difficulty.  In particular, the cube varieties that make up the  $\cubed{(n-2)}$ interior of an  $\cubed{n}$ cube can be arbitrary.  The same is true for the $6(n-2)^2$ cubes in the interior of the six faces, since all six colors are represented on every colored cube.   The only cubes that cannot be chosen arbitrarily are those that lie on the edges, which we call the \textit{frame} of the $\cubed{n}$ cube.   Therefore, once a cube's frame has been constructed, the Colored Cubes problem is solved.   The focus of this paper is Conjecture 5.4 in \cite{BerkoveEtAl}, which posits that for $n$ sufficiently large, when there are enough cubes to fill in a frame then one can actually construct it.  Our main results are:

\medskip

\noindent \textbf{Theorem A}. 
Any set of $24$ cubes is sufficient to build a $\cubed{2}$ and a $\cubed{3}$ frame.  This is the best possible result.

\medskip

\noindent \textbf{Theorem B}. 
When $n >3$, given an arbitrary collection of $12n-16$ colored cubes, one can always build a frame.

\medskip

An important component of a number of our proofs involves the obvious $S_6$ action on the cubes induced by color permutation.  There is an arrangement of the collection of $30$ colored cubes--the tableau--where the action has particularly nice properties.  This action allows us to simplify many our arguments, and plays a central role is a number of proofs in this paper.  As a bonus, the $S_6$ action on the tableau provides another concrete demonstration of  the action of an outer automorphism of $S_6$.  

In the next section, we introduce notation and a formal statement of the Colored Cubes problem.  In the third section, we describe the tableau, give some of its properties, and discuss its associated $S_6$ action.  In the last three sections we use the $S_6$ action and other results to prove our main theorems.  We conclude with some open questions.

It is a pleasure to thank Liz McMahon and Gary Gordon for many helpful conversations during the construction of these arguments, and for comments and suggestions which improved the overall quality of this paper.

\section{Definitions and Problem Statement}
To distinguish between different cubes, we set up an equivalence and say that cubes that have the same coloring up to isometry of $\mathbb{R}^3$ are of the same \textit{variety}.    As mentioned in the introduction, there are $30$ distinct varieties of colored cubes.   A \textit{solution} to the puzzle is an arrangement of $n^3$ cubes that forms an $n \times n \times n$ cube where each face is one color.  A \textit{corner solution} is a set of eight cubes that forms the corners of a solution.   We say that an $n \times n \times n$ solution cube is {\it modeled} after a colored cube if the solution and its model are of the same variety.  We note that a solution and its corner solution are also  always of the same variety.

We also track relative positions of face colors on the cube.  The unordered pair of colors that are opposite each other on a cube form an \textit{opposite pair}.  Similarly, the unordered pair of colors on faces that share an edge of a cube are an \textit{adjacent pair}.  We call the three colors on the faces which meet at the corner of cube a \textit{corner triple}.  Since order matters in this last case, we read the colors clockwise around the corner.  Furthermore, two corner triples are equivalent if one is a cyclic permutation of the other.  That is, $(2, 3, 6) \sim (6, 2, 3) \not \sim (6, 3, 2)$.  Corner triples are {\it mirror image} corner triples if they contain the same colors but are not equivalent.  Note that on a given colored cube, every possible unordered pair of distinct colors appears either as an adjacent pair or an opposite pair.  So every colored cube contains $12$ of the $15$ possible adjacent pairs, and $8$ of the $40$ possible corner triples.

In general, we will denote a generic cube by $c$ and its \textit{mirror image} by $c\mir$; $c$ and $c\mir$ together form a pair of \textit{mirror cubes}.  We note that  mirror cubes have the same opposite and adjacent pairs, but have mirror image corner triples.  We have the following useful characterization.  

\begin{lemma}\label{lem:mirror}
The varieties $c$ and $c\mir$  are related through the exchange of an opposite pair.  Therefore,  the varieties $c$ and $c\mir$ are the only two with the same opposite pairs. 
\end{lemma}

\begin{proof}
This is directly seen by visualizing a mirror placed parallel to some face of a cube $c$.    For the second part, iterate the mirror procedure.
\end{proof}

As mentioned in the introduction, the Colored Cubes puzzle is solved once the frame is complete.  In an $\cubed{n}$ cube, this consists of $12(n - 2)+8 = 12n - 16$ cubes.  The solution to the Colored Cubes Puzzle in \cite{BerkoveEtAl} involved the construction of a frame given an arbitrary collection of $n^3$ cubes.  Since the number of cubes in the frame grows linearly with $n$, starting with $n^3$ cubes is generous.  In fact, since there are only $30$ distinct varieties of colored cubes, when $12n -16  < \frac{n^3}{30}$, the pigeonhole principle guarantees a solution for the frame using just one variety of cube.  This happens for $n \geq 19$.  We therefore focus our attention towards determining the value of the \textit{frame function} $\fr(n)$, which gives the size of the smallest collection of cubes where one is guaranteed to be able to build a frame, regardless of the component cube varieties.   Knowing the value of $\fr(n)$ gets to the heart of the Colored Cubes puzzle, as it is the best possible result.

\section{The cube tableau and automorphisms of $S_6$}\label{section:Aut}

We describe an arrangement of the 30 colored cube varieties, the tableau, which will aid in the proofs that follow.  This tableau, which can be found in the appendix, is an arrangement of the 30 cube varieties in a $6 \times 6$ matrix.  The six diagonal entries are blank, and the other $30$ slots each contain the net of one of the cube varieties.  Under each net are three pairs of numbers in parentheses, which represent opposite faces of the cube.  The collection of the $30$ distinct varieties of colored cubes contains a great deal of structure, much of which is visible in the tableau.

\begin{enumerate}
\item  There are $15$ combinations of six colors into three sets of two, and all $15$ are represented in the nets above the diagonal and in the nets below the diagonal. 
\item Each of the 30 possible color pairs appears exactly once in each row and each column.  Geometrically, this means that two colors are an opposite pair in exactly one cube per row and exactly one cube per column. 
\item The cubes exhibit mirror symmetry across the diagonal line of the tableau.  This can also be seen by Lemma \ref{lem:mirror} since mirror cubes have the same opposite pairs.
\end{enumerate}
       
There is a natural $S_6$ action on the tableau that arises from permuting the set of six face colors.   Furthermore, the action provides a way to visualize the outer automorphism of $S_6$.  Using terminology that goes back to Sylvester \cite{Sylvester}, call an unordered pair of six colors a \textit{duad} and a collection of three duads a \textit{syntheme}.  A collection of five synthemes such that each of the $15$ duads appears exactly once is a \textit{pentad}.  Duads, synthemes, and pentads correspond to opposite pairs, the three opposite pairs in a cube, and the collection of such pairs in each row and column of the tableau, respectively.  

\begin{lemma} \label{lem:rowfix}
A permutation of the color palette induces a permutation that sends rows to rows and columns to columns.  Once the permutation is determined on the top row of the tableau, the action on the rest of the tableau is uniquely determined.
\end{lemma}

\begin{proof}
A color permutation of the palette takes pentads to pentads, so it is sufficient to show that a row isn't taken to a column.  Since all permutations in $S_6$ are generated by transpositions, we will show that applying a transposition sends rows to rows.  Let  $(a_1, a_2)$ be the transposition.  There is precisely one cube per row with the duad $\{a_1, a_2\}$ as part of its syntheme, so call it cube $c$.  By Lemma \ref{lem:mirror}, exchanging colors for $c$ has the geometric effect of sending it to its mirror image $c\mir$.  That means that $c$'s row pentad is sent either to the row pentad of $c\mir$ or the column pentad of $c\mir$.   Take another cube $d$ in $c$'s row with syntheme $\{a_1 a_3\}\{a_2 a_4\}\{a_5 a_6\}$.  Applying $(a_1, a_2)$ to $d$ yields the syntheme $\{a_2 a_3\}\{a_1 a_4\}\{a_5 a_6\}$.  Since $d\mir$ is part of $c\mir$'s column pentad and already contains the duad $\{a_5a_6\}$,  $(a_1, a_2)$ must send $d$ to $c\mir$'s row pentad.  Finally, since the tableau has mirror symmetry across the diagonal, once the action on the rows is determined, so is the action on the columns.
\end{proof}

\begin{lemma}\label{lem:stabilizer}
Consider the set of permutations of the color palette that fix a distinguished variety in the top row of the tableau.  This set is isomorphic to $S_4$, and it acts faithfully on the non-distinguished cubes in the top row.
\end{lemma}

\begin{proof}
The stabilizer of the distinguished cube is its group of direct isometries in $\mathbb{R}^3$, which is $S_4$, and by Lemma \ref{lem:rowfix}, any palette permutation that fixes a the distinguished variety sends its row to itself.  A simple check shows that at least one element of the stabilizer acts non-trivially on the tableau.  Further, if a permutation fixes every cube in the first row, by Lemma \ref{lem:rowfix} it fixes every cube in the tableau and must be in the tableau stabilizer, which is a normal subgroup of $S_6$.  The only possible  non-trivial normal subgroup, $A_6$, is clearly too large to be this stabilizer.  We conclude the action is faithful.
\end{proof}

\begin{corollary}
Let $\sigma$ be the map that takes $\alpha \in S_6$ to its action on the six pentads via permutation of the  color palette.  Then $\sigma$ is an outer automorphism of $S_6$.
\end{corollary}

\begin{proof}
Given a transposition $(a_1, a_2)$, since every row and column in the tableau contains exactly one cube with the duad $\{a_1, a_2\}$, each of these cubes gets sent to a different row, the row containing its mirror image.  Therefore, the effect of $(a_1, a_2)$ on the tableau is to swap three pairs of rows.  This map isn't an inner automorphism, since it doesn't preserve cycle structure.  
\end{proof}

The construction of the tableau and its relation to automorphisms of $S_6$ were previously known.  J. H. Conway used the tableau to provide a complete answer to Puzzle \ref{Prob2} in the Introduction \cite{Conway}.  In Conway's notation, tableau rows are labeled $A$ through $F$ and tableau columns are labeled $a$ through $f$.  The cubes in the tableau then have the following ``coordinates.''

\begin{center}
\begin{tabular}{c c c c c c}
& Ab & Ac & Ad & Ae & Af \\ 
Ba & & Bc & Bd & Be & Bf \\ 
Ca & Cb & & Cd & Ce & Cf \\ 
Da & Db & Dc & & De & Df \\ 
Ea & Eb & Ec & Ed & & Ef \\
Fa & Fb & Fc & Fc & Fe & 
\end{tabular}
\end{center}
In this nomenclature, cubes Xy and Yx form a mirror pair.  

In addition, P. Cameron mentions the connection between the tableau and automorphisms of $S_6$ in a WordPress blog \cite{Cameron}.  The tableau provides a particularly nice demonstration of the action of the outer automorphism of $S_6$, and we believe it deserves to be better known.  

There are additional relationships between varieties in the rows and columns which we will use in subsequent sections.  Before we state these results we recall two lemmas from \cite{BerkoveEtAl}.

\begin{lemma} \label{lem:sidefacts} \cite[Lemma 2.6]{BerkoveEtAl}
Two colored cubes share exactly nine, ten, or twelve adjacent pairs according to whether they share exactly zero, one, or three opposite pairs, respectively.
\end{lemma}

\begin{lemma}  \label{lem:cornerfacts} \cite[Lemma 2.7]{BerkoveEtAl}
Given two varieties:
\begin{enumerate}
\item If they have no opposite pairs in common, then they share zero or two corner triples.
\item If they have one opposite pair in common, then they share exactly two corner triples.
\item If they have three opposite pairs in common, then they share zero or eight corner triples.
\end{enumerate}
\end{lemma}

The proof of Lemma \ref{lem:cornerfacts} follows from the description of how a variety $c$ is related to the $29$ other cube varieties, which we summarize here. There are eight varieties that arise from choosing a corner triple of $c$ and cyclically permuting its colors clockwise and counterclockwise.  These varieties share two corner triples but no opposite pairs with $c$.   Their mirror images form eight new varieties that share no corner triples and no opposite pairs with $c$.   There are twelve varieties that come from exchanging an adjacent pair on $c$.  These varieties all share two adjacent corners and one opposite pair with $c$.  The last variety is $c\mir$, which shares three opposite pairs and no corner triples with $c$.  This characterization also provides additional structure to the tableau, as well as a means for its construction.

\begin{corollary} \label{cor:tableaustructure}
Fix a distinguished variety $c$ in position Ab in the tableau.  Then
\begin{enumerate}
\item The variety Ba is $c\mir$.
\item The eight varieties related to variety $c$ by cyclically permuting  of a corner followed by a mirror image are in column b and row A.  
\item The eight varieties related to variety $c$ by cyclically permuting of a corner are in column a and row B.  
\item The twelve varieties related to variety $c$ by edge flipping are in columns c, d, and e, and rows C and D, and E.
\end{enumerate}
\end{corollary}

\begin{proof}
Since the tableau has mirror symmetry, $c\mir$ is in position Ba.  Column b and row A are pentads containing variety $c$.  From the characterization in Lemma \ref{lem:cornerfacts} and the discussion afterwards, these eight varieties must be formed from $c$ by cyclically permuting a corner followed by a mirror image.  By checking case, one finds that the varieties that are formed using a clockwise cyclic permutation constitute one pentad, and the varieties that are formed using a counterclockwise cyclic permutation form the other.  Using mirror symmetry, column a and row B are formed from $c$ by cyclically permuting a corner.  That leaves the last twelve varieties in the claimed positions.
\end{proof}

\begin{corollary} \label{cor:cornershare}
Let Xy be an arbitrary variety in the tableau.  Then the varieties that share no corners with variety Xy are precisely variety Yx, the varieties in column y, and the varieties in row X.
\end{corollary}

\begin{proof}
The variety Yx is the mirror image of Xy, so shares no corner triples with it.  The statement about the row and column follows from Lemmas \ref{cor:tableaustructure} and  \ref{lem:rowfix} once we note that there is a permutation of the color palette that moves Xy to the distinguished position.  
\end{proof}

\section{Building a $\cubed{2}$ solution: $\fr(2) = 24$}\label{sec:2frame}

In \cite{BerkoveEtAl} it was (incorrectly) conjectured that $\fr(2) = 23$.  Using properties of the tableau from Section \ref{section:Aut}, one can quickly construct a counterexample.  We first describe a way to keep track of the cubes in an arbitrary set using the tableau.  A set may contain many cubes of one variety; if variety Xy occurs $k$ times in the set, put $k$ into Xy's position in the tableau.  We call $k$ the \textit{repetition number} of the variety, and note that if $k \geq 8$, then there is a corner solution modeled after Xy.  If variety Xy is not part of the set, put a dot in its position in the tableau.   

Consider the following collection of $23$ cubes.  Within the five cubes of one row pentad, take seven copies of three varieties and one copy of the other two.   By Lemma \ref{lem:rowfix} there is a color automorphism that moves one variety with multiplicity seven to position Ab in the tableau.  By Lemma \ref{lem:stabilizer} there is another automorphism that puts the row A into the following form:
\begin{center}
\begin{tabular}{c c c c c c}
& 7 & 7 & 7 & 1 & 1 \\ 
\end{tabular}
\end{center}
This collection cannot be used to build any variety in the top row of the tableau since no two varieties in this row share corners.  One can't build a variety Xa in column a, since Lemma \ref{lem:cornerfacts} implies that the other varieties in row A can contribute at most two corners to the frame,  variety Xa shares no corners with variety Ax, and at least one variety only appears once in the collection.  In a similar way, a variety Xy in rows b through f shares no corners with variety Ay.  The remaining four cube varieties are insufficient to construct a corner solutions.  

We claim that $\fr(2) = 24$.  Our argument uses partitions of the set of $24$ cubes into $k$ subsets, where the size of a subset determines a variety's repetition number.  We write our partitions in decreasing order as like $54^3321^2$.  This tells us that the set of $24$ cubes is divided into eight distinct varieties, one variety with repetition number $5$, three with repetition number $4$, etc.   Once we fix a partition, we still need to assign varieties to the repetition numbers, and with a set of 30 possibilities the number of ways to do this is quite large.   For example, the number of cases we have to check to ensure that the partition $54^3321^2$ of $24$ cubes always has a solution is about $1.97 \times 10^{10}$.  Instead of checking them all, we determine conditions on the structure of the partition which will always yield a solution.  This is the focus of the following computer calculations and lemmas.

\begin{lemma}\label{lem:10cases}
Any collection which contains ten varieties always has a corner solution.
\end{lemma}        

\begin{proof}
The proof is by computer search, written in \textit{Mathematica}.  One builds the \textit{cube-corner} bipartite graph, where one set of vertices is the set of 30 possible colored cubes and the other set is the 40 possible corners.  Edges in the cube-corner graph connect varieties in the first vertex set with their corners in the second.   The algorithm proceeds by choosing a subset of ten varieties from the 30 possible.  Rather than work with the entire bipartite graph each time, for each subset the algorithm loops through all 30 possible corner solutions, building the subgraph of the cube-corner graph where one set of vertices is the subset of ten varieties, and other set is the set of eight corners from the variety of the potential corner solution.  \textit{Mathematica} then determines a maximal matching, where a matching of size eight means the distinguished corner solution can be assembled from the subset.
\end{proof}

At first glance, Lemma \ref{lem:10cases}'s proof requires checking ${30 \choose 10} \approx 3.0 \times 10^7$ cases, but we can use the tableau to reduce this number some.  Given ten distinct varieties in the tableau, at least two are in the same row.  This row can be moved to the top of the tableau using the $S_6$ action, and furthermore we can assume that the two varieties are Ab and Ac.  This leaves ${28 \choose 8} \approx 3.1 \times 10^6$ cases, about one tenth of the previous number.  This shorter computation still requires almost four hours of CPU time using a 2.70 GHz Intel Core i7-3740QM CPU with a 64-bit operating system and 8 GB of memory.  

The result in Lemma \ref{lem:10cases} implies that it is sufficient to consider subsets of $24$ cubes comprised of between four and nine different varieties.  This results in $354$ partitions to consider.  We use the next lemma, proven by computer search, to show that the vast majority of these contain a subset that forms a corner solution, regardless of the varieties which appear.

\begin{lemma}\label{lem:computedcases}
Given the following sets of cubes, one can always construct a corner solution.
\begin{itemize}  
\item 18 cubes consisting of two varieties with repetition number $7$ and any four other varieties. 
\item 16 cubes consisting of two varieties with repetition number $6$ and two varieties with repetition number $2$.
\item 19 cubes consisting of two varieties with repetition number $5$ and three varieties with repetition number $3$.
\item 16 cubes consisting of three varieties with repetition number $4$ and two varieties with repetition number $2$.
\item 18 cubes consisting of six varieties with repetition number $3$.  
\item 14 cubes consisting of seven varieties with repetition number $2$.
\end{itemize}
\end{lemma}

We note that the corresponding Lemma 4.1 in \cite{BerkoveEtAl} erroneously claimed that one can always construct a corner solution from two varieties with repetition number $4$ and four with repetition number $2$.  The correct statement is above, that three varieties with repetition number $4$ and two with repetition number $2$ suffices.  Interested readers can contact the first author for copies of the code that was used to prove Lemmas \ref{lem:10cases} and \ref{lem:computedcases}.

Lemma \ref{lem:computedcases} immediately implies the existence of a corner solution in all but $22$ of the partitions.  (One partition for which the lemma does not apply, for example,  is $753^221^4$.)  The next two lemmas provide tools to handle many of these remaining cases. 

\begin{lemma}\label{lem:4of5}
Given a set of eight cubes consisting of four varieties from any pentad, each with repetition number $2$, one can always construct a corner solution.
\end{lemma}

\begin{proof}
Taking mirror images if necessary, we can assume the pentad is a row pentad.  Then by Lemmas \ref{lem:rowfix} and \ref{lem:stabilizer} we can assume that the varieties are in the top row (row A) and in columns b through e of the tableau, like so:
\begin{center}
\begin{tabular}{c c c c c c}
& 2 & 2 & 2 & 2 & $ \cdot$ \\ 
\end{tabular}
\end{center}
By Corollary \ref{cor:cornershare}, each of these varieties shares exactly two corners with varieties Bf, Cf, Df, or Ef.  Furthermore, the corners are distinct, since the top row is a pentad.  So there are at least four possible corner solutions.
\end{proof}

\begin{lemma}\label{lem:technical}
Given a decreasing partition $(a_1, a_2, \ldots, a_n)$ with $a_2 \geq 4$, if $a_1 = 7$ and $a_2 + n \geq 13$, then there is a subset of $8$ cubes that forms a corner solution.  The same is true if $a_1 = 6$ and $a_2 + n \geq 14$.
\end{lemma}

\begin{proof}
For simplicity, we will refer to the variety with repetition number $a_i$ as variety $i$.  By using the lemmas from Section \ref{section:Aut}, we can assume that variety 1, which appears $7$ times, is in position Ab in the tableau.  Then by Corollary \ref{cor:cornershare}, there is automatically a solution unless the remaining cube varieties are among the nine in row A, column b, and the mirror variety Ba.  Therefore, variety $2$ is either in the same pentad as $c$, or is $c^*$.  By using mirror symmetry and/or a color automorphism, we may assume that variety $2$ is in position Ac or Ba as in the diagrams below.  Boxes represent the cube varieties that don't automatically result in a corner solution as noted in Corollary \ref{cor:cornershare}.

\medskip

\noindent\begin{minipage}[b]{.5\textwidth}
\begin{center}
\begin{tabular}{c c c c c c}
             & $a_1$  & $a_2$    & $\Box$   & $\Box$   & $\Box$ \\ 
$\Box$  &             & $\cdot$  & $\cdot$  & $\cdot$ & $\cdot$ \\ 
$\cdot$ & $\Box$ &               & $\cdot$  & $\cdot$ & $\cdot$ \\ 
$\cdot$ & $\Box$ & $\cdot$ &                & $\cdot$ & $\cdot$ \\ 
$\cdot$ & $\Box$ & $\cdot$ & $\cdot$  &               & $\cdot$ \\
$\cdot$ & $\Box$ & $\cdot$ & $\cdot$  & $\cdot$  & 
\end{tabular}
\vskip .1in

Case Ac
\end{center}
\end{minipage} 
\hfill
\begin{minipage}[b]{.5\textwidth}
\begin{center}
\begin{tabular}{c c c c c c}
             & $a_1$  & $\Box$    & $\Box$   & $\Box$   & $\Box$ \\ 
$a_2$   &             & $\cdot$  & $\cdot$  & $\cdot$ & $\cdot$ \\ 
$\cdot$ & $\Box$ &               & $\cdot$  & $\cdot$ & $\cdot$ \\ 
$\cdot$ & $\Box$ & $\cdot$ &                & $\cdot$ & $\cdot$ \\ 
$\cdot$ & $\Box$ & $\cdot$ & $\cdot$  &               & $\cdot$ \\
$\cdot$ & $\Box$ & $\cdot$ & $\cdot$  & $\cdot$  & 
\end{tabular}
\vskip .1in

Case Ba 
\end{center}
\end{minipage}


\smallskip

\noindent \textbf{Case Ac}:  Variety $2$ shares two corners with variety $1^*$ and all the varieties in the column's boxed positions.  We claim that if $k \leq 4$ of these varieties are represented, then they can be used to construct $k$ distinct corners of a corner solution modeled after variety $2$.  Since there are only five slots in the top row of the tableau, we can always build a corner solution when $a_2 + (n - 5) \geq 8$.  

The proof of the claim is contained in Lemma 4.4 in \cite{BerkoveEtAl}, but here's a direct proof.  Consider a multigraph whose vertices are the corners of variety $2$.  Given a variety that shares two corners with variety $2$, add an edge to the graph between the vertices representing those corners.  By the discussion after Lemma  \ref{lem:cornerfacts}, graph edges are either edges of variety $2$ or long diagonals (the latter may occur twice).  The shortest possible cycle made up of these edges has length $4$.

\smallskip

\noindent \textbf{Case Ba}:  Here, variety $2$ shares two corners with every boxed variety, so $a_2 + (n - 2) \geq 8$ varieties guarantees a corner solution.  This condition holds when $a_2 + n \geq 13$.

\smallskip

Finally, if $a_1$ = 6, then one variety with repetition number $1$ can be in the dotted positions in the tableau without yielding a corner solution.  Now there are $7$ distinct corners available to build a corner solution modeled after variety 1, and we're in the prior case.
\end{proof}

We apply Lemma \ref{lem:technical} to the $22$ partitions, and end up with eight remaining cases to consider, listed by increasing number of varieties:
\[ 
75^31^2, \  75^241^3, \ 65^31^3, \ 743^321^2, \ 65^241^4, \ 5^41^4, \ 73^421^3, \ 5^341^5
\]
These cases can all be handled by ad hoc methods, although broadly speaking the techniques are similar to those in the proof of Lemma \ref{lem:technical}.  One places variety $1$ is in position Ab.  Although variety $2$ can be in either position Ac or Ba, we usually only need to check position Ac, since that case, as in Lemma \ref{lem:technical}, is the most involved.  Also, Lemma \ref{lem:4of5} implies that if there are four varieties in the same pentad with repetition number greater than 1 then there is a solution, so we avoid that situation too.  Keeping these observations in mind, we sketch arguments for the three most subtle of the remaining cases.  The other five cases are similar, but easier.

\begin{enumerate}
\item The partition $743^321^2$: Let $a_1 = 7$ and $a_2 = 4$ as in the Case Ac.  Using Lemma \ref{lem:4of5}, we can assume that there is one variety with repetition number $3$ in the row pentad and the other two varieties with repetition number $3$ are in the column pentad, say Cb and Db.  Variety $2$ also shares two corners with Cb and Db, and since the varieties in the column form a pentad these corners are distinct.  That is, these two varieties contribute four corners to a corner solution modeled after variety $2$.  
\item The partition $73^321^3$: This case is similar to the one above, although there are essentially two cases to consider, representatives of which are shown below.  The final variety that occurs once can be in either boxed position.
\vskip .1in
\noindent\begin{minipage}[b]{.4\textwidth}
\begin{center}
\begin{tabular}{c c c c c c}
             &     $7$  &    $3$     &     $3$    &     $1$    & $\Box$ \\ 
   $2$    &             & $\cdot$  & $\cdot$  & $\cdot$ & $\cdot$ \\ 
$\cdot$ &  $3$   &                 & $\cdot$  & $\cdot$ & $\cdot$ \\ 
$\cdot$ &  $3$   & $\cdot$   &                & $\cdot$ & $\cdot$ \\ 
$\cdot$ &  $1$   & $\cdot$ & $\cdot$  &               & $\cdot$ \\
$\cdot$ & $\Box$ & $\cdot$ & $\cdot$  & $\cdot$  & 
\end{tabular}
\vskip .1in

Case 1
\end{center}
\end{minipage} 
\hfill
\begin{minipage}[b]{.4\textwidth}
\begin{center}
\begin{tabular}{c c c c c c}
             &     $7$  &    $3$     &        $3$   &    $1$     & $\Box$ \\ 
   $3$    &             & $\cdot$ & $\cdot$    & $\cdot$  & $\cdot$ \\ 
$\cdot$ &  $3$    &               & $\cdot$    & $\cdot$  & $\cdot$ \\ 
$\cdot$ &  $2$    & $\cdot$  &                 & $\cdot$  & $\cdot$ \\ 
$\cdot$ &  $1$    & $\cdot$  & $\cdot$    &               & $\cdot$ \\
$\cdot$ & $\Box$ & $\cdot$ & $\cdot$   & $\cdot$  & 
\end{tabular}
\vskip .1in

Case 2
\end{center}
\end{minipage}

\noindent In both cases, variety $2$ in position Ac has repetition number $3$, and variety Ba has repetition number $2$ or $3$.  In both cases, the three varieties in the same column as variety 1 are part of a pentad.  Therefore, those three varieties can contribute $5$ corners towards a corner solution modeled after variety $2$.

\item The partition $5^341^5$: Variety $2$ has repetition number $5$, and there must be at least three varieties in the column pentad of variety 1.  These three contribute three distinct corners to a corner solution modeled after variety $2$.
\end{enumerate}  

Since there are corner solutions for all possible partitions we have the main result of this section and the first part of Theorem A.

\begin{theorem}[Theorem A, part 1]\label{thm:2soln}
Given any set of $24$ colored cubes, there is always a subset from which one can construct a corner solution.  Consequently, $\fr(2) = 24$.
\end{theorem}

\section{Building a $\cubed{3}$ solution:  $\fr(3) = 24$}\label{sec:3frame}

In contrast to the $\cubed{2}$ case, determining $\fr(n)$ for $n > 2$ requires knowledge of more than the corner solution.  This is reflected in the proofs of these cases, which have a different flavor than the case $n = 2$.  In order to determine the way to place cubes into the interior of the frame, we need to know a bit more about how cubes varieties are related to each each.   This is the subject of the next few results, which describe how to construct partial frames given cubes of a particular type, how to place cubes  into edges of the frame, and how cubes that share an opposite pair are related.  We use these results both in this section and in Section \ref{sec:framebuild}.

\begin{lemma} \label{lem:extension}
Given a corner solution and $k(n-2)$ cubes that share $k$ edge pairs with the corner solution, then it is possible to place all  $k(n-2)$ cubes into the $n$-frame.
\end{lemma}

\begin{proof}
In light of Lemma \ref{lem:sidefacts}, $k = 9, 10$, or $12$.  The result is clear when $k = 12$.  If $k < 12$, assume that in the construction of the frame, fewer than $k$ edges are complete.  Since any single cube shares at least $k$ edges with the corner solution, it can be fit into the cube.  This process can always be repeated when fewer than $k$ edges are complete in the frame, which requires $k(n-2)$ cubes. 
\end{proof}

The next results examine the situation where one has a collection of cubes that all share one opposite pair.

\begin{lemma} \label{lem:sixvar}
There are exactly six distinct cube varieties that share any opposite pair.  Given any three of the six, all adjacent pairs on one variety can be found on either one or both of the remaining two varieties.
\end{lemma}

\begin{proof}
Once one opposite pair has been determined, there are three ways to partition the remaining four colors into two opposite pairs, and there are two mirror varieties, $c$ and $c\mir$, for each set of three opposite pairs.  

For the second statement in the lemma, if two of the varieties are mirror images, the result follows since these varieties have the same adjacent pairs by Lemma \ref{lem:sidefacts}.   Otherwise, Lemma \ref{lem:sidefacts} implies that two distinct varieties share ten adjacent pairs, so the two together have $24 - 10 = 14$ distinct adjacent pairs.  Since all the cubes share one opposite pair, these represent all possibilities.
\end{proof}

\begin{lemma} \label{lem:opppairfacts} 
Fix $c$, one of the six cube varieties that share one opposite pair.  The variety $c$ shares no corners with $c\mir$, and shares exactly two unique corners with each of the other four varieties.
\end{lemma}

\begin{proof}
Assume that the opposite pair are the colors $5$ and $6$, and that the colors around the girth of $c$ are given by the cyclic permutation $(1234)$.  The cyclic permutations for the other five varieties are $(2134)$, $(1324)$, $(1243)$, $(4231)$, and $(4321)$.  In the girth, the first four share exactly one adjacent pair with $c$, the edges $\{3,4\}$, $\{4,1\}$, $\{1,2\}$, and  $\{2,3\}$ respectively.  These give rise to a pair of corners that match $c$'s.  The last cube is $c\mir$.
\end{proof}

\begin{corollary} \label{cor:makecorner}
Take four copies of $c$ and two copies each of two other varieties that share the same opposite pair.  If the two other varieties are not $c\mir$, then the eight cubes can be assembled into a corner solution modeled on $c$.  There is a similar result using six copies of $c$ and two cubes that are not of the same variety as $c\mir$.
\end{corollary}

When a cube and its mirror image occur with large repetition number we can say even more.

\begin{lemma} \label{lem:two7s}
Let $S$ be a set consisting of seven cubes each of varieties $c$ and $c\mir$ along with one cube of any other variety.  Then $S$ has a subset of eight cubes that forms a corner solution modeled on either variety $c$ or variety $c\mir$.
\end{lemma}

\begin{proof}
In the tableau, assume that $c$ is the distinguished variety.  By Corollary \ref{cor:cornershare}, any variety that isn't $c$ or $c\mir$ shares two corners with either $c$ or $c\mir$.
\end{proof}

\begin{lemma}  \label{lem:makecornermin}\cite[Lemma 3.2]{BerkoveEtAl}
It is always possible to find a corner solution given 15 cubes that share one opposite pair.
\end{lemma}

\begin{proof}
A sketch is as follows: Find the variety, $c$, that occurs with largest repetition number in the set of $15$.  If $|c|=7$, then a set of seven copies each of $c$ and $c\mir$ does not have a subset that forms a corner solution, but the set formed by adding any other variety does by Lemma \ref{lem:two7s}.  The cases of smaller $|c|$ are similar, and have lower thresholds.    
\end{proof}

We showed that $\fr(2) = 24$ in Section \ref{sec:2frame}.  Since $20$ cubes are required to build the $\cubed{3}$ frame, we also have that $\fr(3) \geq 24$.  We now show that this inequality is sharp, which is the second part of Theorem A.  

We base our proof on the technique used in  Theorem 3.4 of \cite{BerkoveEtAl}.

\begin{theorem}[Theorem A, part 2] \label{thm:3soln}
The frame of a $\cubed{3}$ puzzle can always be completed given $24$ arbitary cubes, so $\fr(3) = 24$.
\end{theorem}
\begin{proof}
Given $24$ cubes, there is always a $\cubed{2}$ solution.  If we can place twelve of the remaining $16$ cubes into the twelve open slots in the frame then we have a solution.  We assume this isn't possible and show that we can find a different corner solution whose frame can be completed.  Call the set of $16$ cubes $S$.   We know from Lemma \ref{lem:extension} that the cubes from $S$ can be fit into at least nine out of the twelve slots in the $\cubed{3}$ frame.  Track these adjacent pairs (edges) individually by labeling the twelve adjacent pairs of the corner solution as $e_1$, through $e_{12}$, and let $c_i$ denote the number of cubes out of the $16$ remaining that have an adjacent pair $e_i$.  We note that $c_i \geq  i$ is a sufficient condition for a solution to the frame--pick any cube that has adjacent pair $e_1$ as its representative in the frame and continue the process in ascending order.  If this cannot be done, then there must be a largest index $j$ with $c_j < j$. 

We bound the total number of adjacent pairs two ways.  On the low end, each of the $16$ cubes in $S$ shares at least nine adjacent pairs with the solution by Lemma \ref{lem:sidefacts}.  On the other hand, the $j$ edges $e_1, \ldots,  e_j$ occur no more than $j - 1$ times each, and the $12 - j$ edges $e_{j+1}, \ldots, e_{12}$ occur as many as $16$ times.  We have the inequalities
\[
16 \times 9 \leq \sum_{i=1}^{12} c_i \leq (j)(j-1) + (12 - j)(16).
\]
When we solve this quadratic inequality over the integers, we find that $j \leq 3$ or $j \geq 15$.  The latter case is impossible as $j \leq 12$.  We conclude that if we cannot complete the frame, it is because at most three edges could not be matched.  Although this is consistent with the results of Lemma \ref{lem:extension}, we now know exactly how we might fail to have a solution.

\medskip

\noindent \textbf{Case 1: $e_1 = 0$}.  In this case, there is one adjacent pair, say $\{1,2\}$, that is opposite in all $16$ cubes in $S$.  By Lemma \ref{lem:makecornermin}, there is a subset of eight cubes that forms a corner solution, and this corner solution must have $\{1,2\}$ as an opposite pair.  We repeat the enumerating process above with this new corner solution.  Each cube shares at least ten adjacent pairs with the corner solution by Lemma \ref{lem:sidefacts}, so 
\[
16 \times 10 \leq \sum_{i=1}^{12} c_i \leq (j)(j-1) + (12 - j)(16).
\]

This implies that $j \leq 2$.  If there is no solution, then there are two new possibilities.  The first is that $c_1 = 0$.  In this case, there is another pair $\{ x,y\}$, which is opposite on all $16$ cubes.  Since the cubes in $S$ already have $\{1, 2\}$ as an opposite pair, they must have the same three opposite faces.  That is, the cubes consist of a variety $c$ and its mirror variety $c\mir$, so they all have the same adjacent pairs.  At least eight of these cubes are of the same variety and form a corner solution.  Using Lemma \ref{lem:extension}, we can place the eight cubes not in $S$ into the $\cubed{3}$ frame.  Remaining cubes from $S$ now complete the solution.

The second possibility is that $c_1 = c_2 = 1$.  Furthermore, we can assume that adjacent pairs $e_1$ and $e_2$ are on the same cube, since if the adjacent pairs are on different cubes then each can be used as its representative in the frame.  As in the previous case, there are at least $15$ cubes that share the same opposite pairs and are of exactly two varieties.  Pick eight identical cubes for the corner solution and proceed as before.
\medskip

\noindent \textbf{Case 2: $e_1 = e_2 = 1$}.  We can again assume that the adjacent pairs $e_1$ and $e_2$ are from the same cube.  In fact, we can also assume that the pairs $e_1$ and $e_2$ have no colors in common.  For say $e_1$ and $e_2 $ are the adjacent pairs $\{x,y\}$ and $\{w,z\}$.  Then $\{x,y\}$ and $\{w,z\}$ must be opposite pairs on the remaining $15$ cubes of S, which can only happen if $x$, $y$, $w$, and $z$ are distinct.  Therefore, these $15$ cubes have the same three opposite pairs.  We now proceed as in Case 1.  

\medskip

\noindent \textbf{Case 3: $e_1 = e_2 = e_ 2 = 2$}.   If these adjacent pairs are from three cubes then we can pick representatives for the three edges and complete the $\cubed{3}$ frame.  If the adjacent pairs are from the same two cubes, then there are $14$ cubes from $S$ that have the same opposite pairs.  If there are eight of one variety we can complete the frame as above.  Otherwise, if there are seven each of a variety $c$ and its mirror image $c\mir$, then take a cube of another variety from $S$ and apply Lemma  \ref{lem:two7s} to build a corner solution modeled on $c$ or $c\mir$.  Start filling in the edges of the frame using the cubes of varieties other than $c$ and $c\mir$, then use varieties $c$ and $c\mir$ to complete the frame.  \end{proof}

\section{Building a $\cubed{n}$ solution:  $\fr(n) = 12n - 16$ for $n \geq 4$}\label{sec:framebuild}
In this section we prove Theorem B by showing that for $n \geq 4$ we can always build a frame with the smallest possible number of cubes.   In \cite{BerkoveEtAl}, it was possible to adjust the argument in Section \ref{sec:3frame} to show that the construction of an $\cubed{n}$ frame was always possible.  Unfortunately, that does not work in our setting, since in \cite{BerkoveEtAl} it was assumed that there were roughly $n^3$ cubes available to build the frame, whereas here we assume that the number only grows linearly with $n$.  

\begin{theorem}[Theorem B]\label{thm:mainA}
If $n \geq 4$, then $\fr(n) = 12n -16$.
\end{theorem}

\begin{proof}
We proceed using induction.  For the base case, we note that the $4$-frame contains $32$ cubes.   By Theorem \ref{thm:3soln}, any collection of $24$ cubes contains a $\cubed{3}$ solution.  Therefore, any collection of cubes that with enough cubes for a $4$-frame contains a subset of $20$ cubes that forms a $3$-frame.

Now assume that we can solve the frame for the $\cubed{(n-1)}$ cube.  The difference between an $(n - 1)$-frame and an $n$-frame is twelve cubes, one for each edge.  If the $12$ cubes can not be inserted to extend the frame, we will show that it is possible to construct another frame using a different corner solution.  Referencing Lemma \ref{lem:extension}, we see that we will be able to fit at least nine of the twelve cubes into the frame.  We proceed by cases.   

\medskip

\noindent \textbf{Case 1: exactly eleven edges of the $n$-frame are complete}.  Assume that the missing adjacent pair on the frame is $\{1,2\}$, and that, in a worst case, the unplaced cube shares only nine edges with the corner solution.  If a cube in one of the corresponding nine completed edges of the frame contains the adjacent pair $\{1,2\}$, move it to the unfilled position and put the unused cube in its place to complete the puzzle.  If such a swap is not possible, then $\{1,2\}$ must be an opposite pair in all $9(n-2)$ cubes as well as in the unused cube. 

Next, assume there is a cube $c$ among the $2(n-2)$ cubes on the other two non-$\{1,2\}$ frame edges that has $\{1,2\}$ as a (hidden) adjacent pair.  Denote $c$'s adjacent pair contribution to the frame by $\{x,y\}$.  If there is another cube in the nine other edges that also has an adjacent pair $\{x,y\}$, use it to replace $c$ and put the original unplaced cube into its slot.  If not, on each of the $9(n-2)$ cubes colors $x$ and $y$ must form an opposite pair.  Since $\{1,2\}$ is already an opposite pair on these cubes, neither $x$ nor $y$ can be colors $1$ or $2$.

By Lemma \ref{lem:cornerfacts}, it follows that all $9(n-2)$ cubes are of two mirror varieties, and have the same adjacent pairs.  Since $n \geq 4$, we can pick eight of one variety for a corner solution.  Start building the frame using the cubes from the two edges and the unfinished edge.  This is possible by Lemma \ref{lem:extension}.  Complete the frame with the remainder of the $9(n-2)$ cubes, any of which can be used for any edge.

The last possibility is that there is indeed no cube among the $2(n-2)$ on the other two edges that has a $\{1,2\}$ adjacent pair, implying that $\{1,2\}$ is an opposite pair on at least $11(n-2)$ cubes and the unplaced cube.  Denote this set of $11n - 21$ cubes by $S$, and the remaining $n - 3$ cubes of the frame by $T$.  Now $11n - 21 \geq 23$ for $n \geq 4$.  Using Lemma \ref{lem:sixvar}, partition $S$ into three subsets, each consisting of a variety and its mirror image (which have the same adjacent pairs).  One of these subsets will have at least eight elements, so at least one cube variety, say $c$, has multiplicity $4$.  Assume first that there are at least two cubes in each of the other two subsets.  Then by Corollary \ref{cor:makecorner}, these can be assembled into a corner solution modeled on $c$.  We note that 
\[
\left \lceil \frac{11n-21}{3} \right \rceil- 4 \geq 2(n-2) \text{ for } n \geq 4,
\]
so after building the corner solution, there are enough cubes with the same adjacent pairs as $c$ to complete any two edges.  Use Lemmas \ref{lem:sidefacts} and \ref{lem:extension} to fill in as much as possible of ten edges of the $n$-frame using cubes from the other subsets and $T$.  Then the last $\lceil \frac{11n-21}{3} \rceil- 4$ cubes can be used in any remaining open position in the $n$-frame. 

Next, assume that there is only one cube in one of the subsets of $S$.  Move that cube to $T$; now $S$ consists of four varieties in two subsets.  Since  $11n - 22 \geq 22$ for $n \geq 4$, there is at least one subset with $11$ cubes and six of some variety, say $c$.  As above,  
\[
\left \lceil \frac{11n-21}{2} \right \rceil- 4 \geq 2(n-2) \text{ for } n \geq 4.
\]
As long as there are two cubes in the other subset, we can apply Corollary \ref{cor:makecorner} to build a corner solution modeled on $c$.  The rest of the construction is as before.  Finally, if there is no more than one cube in the other subset, then $11n - 23$ cubes are one of two varieties.  Use the one that appears most frequently for the corner solution, and the construction of the $n$-frame is straightforward.

\medskip

\noindent \textbf{Case 2: exactly ten edges of the $n$-frame are complete}.   There are two cubes of the twelve that remain to be placed.  Note that since each cube cannot be placed into the edges, the colors that are adjacent in the unfilled places of the frame are opposite pairs in the two cubes.   This also implies that the two unfilled adjacent pairs do not share a common color.  In a worst case the two cubes share the same nine edges with the corner solution.  If some cube in one of the nine corresponding edges of the frame also has one of the missing adjacent pairs, then move it to the unfilled position and put one of the two unused cubes in its place.  We are now in the situation in Case 1, so there is a solution. 

If such a swap cannot be done, then none of the $9(n-2)$ cubes from the completed edges nor the two used cubes have the two unfilled adjacent pairs in the corner solution.  These $9n - 16$ cubes must therefore share the same two, and hence three, opposite pairs.  Call the set of these cubes $S$, and the remaining $3n - 8$ cubes in the frame the set $T$.  We note that each cube in $S$ is one of two varieties of mirror cubes, and all the cubes in $S$ share the same adjacent pairs. Since $9n - 16 \geq 20$ for $n \geq 4$, at least ten of the  $9n - 16$ cubes are identical cubes.  Pick eight of these to make the corner solution, then start filling in the edges of the frame using cubes from the set $T$.  Since $|T| < 9(n - 2)$, by Lemma \ref{lem:extension} all the cubes in $T$ can be placed into the new frame.  The cubes from $S$ share the same adjacent pairs as the corner solution, so they can be used to complete the $(n+1)$-frame. 
\medskip

\noindent \textbf{Case 3: exactly nine edges of the $n$-frame are complete}.   There are three cubes of the twelve that remain to be placed.  Since none of the three fit into the existing frame, by Lemma \ref{lem:sidefacts} all three share the same nine adjacent pairs with the $n$-frame.  As in the prior cases, if some cube in the corresponding edges of the frame can be used to complete an unfilled edge, swap it out and put one of the three unused cubes in its place.  We are now in Case 2, so there is a solution.

If no such swap is possible, then the $9(n-2)$ cubes from the completed edges in the frame and three unused cubes share the same three opposite pairs.  We call the set of these $9n - 15$ cubes $S$, the remaining $3n - 9$ cubes in the frame the set $T$, and proceed as in Case 2.  
\end{proof}

\section{Open Questions and Final Remarks}
Although MacMahon's original questions are now nearly 100 years old, they are still generating fruitful problems.  In this section we describe a number of open questions for the interested reader to pursue.  We start with the tableau, whose associated $S_6$ action made it very useful in reducing the number and type of cases we needed to consider in this paper.  We believe that there is additional structure in the tableau still to be discovered that would further reduce the amount of computation required to complete the arguments.  This motivates our first question.

\begin{problem}
Refine the analysis of the $S_6$ action on the tableau, and determine which of the results in Lemma \ref{lem:computedcases} follow from this finer understanding.
\end{problem}

The Colored Cubes Puzzle whose solution is in this paper is just one member of a larger family of related puzzles.  A pretty generalization in the spirit of MacMahon's Problem \ref{Prob2} in the introduction is determine the minimum number of cubes required to solve the $\cubed{3}$ puzzle so that all of the internal faces also have matching colors.  We expect that the $\cubed{n}$ version of this problem would be very challenging, but even asymptotic bounds on the number of cubes would be interesting.  

Another way to generalize the problem is by changing the number of colors.  For example:
\begin{problem}
For $n > 1$, determine $g(n,k)$, the minimum number of cubes colored with $k$ colors required to solve the $\cubed{n}$ Colored Cubes Problem.
\end{problem}
There are two variations of this problem, depending on whether $k < 6$ or $k > 6$.  When $k < 6$, one might start by assuming a regularity condition, and say that a cube is \textit{k-colored} if each cube face has a single color and all $k$ colors appear on at least one face of the cube.  In this case, a solution to the frame implies a solution for the puzzle, and $g(n,k)$ is the same as the function $\fr(n,k)$, the analog of $\fr(n)$.  The authors of this paper have completed the calculations for $k = 2$ and $k = 3$ \cite{BCCK}.  The cases of $n = 4$ and $n = 5$ are more challenging, in large part because of the large number of distinct cubes.   The total number of distinct cubes up to rigid rotation can be determined using a Polya counting argument, and are given in the following table for $k \leq 6$.  

\begin{center}
\begin{tabular}{| l |  c |  c | c |  c | c  | }
\hline
Number of Colors &  2  & 3 & 4 & 5 & 6 \\ 
\hline
Distinct Cubes & 8  &  32 &  68  &  75  &  30   \\
\hline
\end{tabular}
\end{center}

\medskip

When $k > 6$, it is no longer possible for all colors to appear on each cube, although we can still apply the regularity condition that no color appear more than once on a face of any cube.  We note, however, that the successful construction of a frame no longer implies that the rest of the $\cubed{n}$ cube can be completed.  We expect that $g(n,k) > \fr(n,k)$ for $k > 6$ and sufficiently large $n$ (probably $n = 2$ or $3$!).  In addition, we also believe that for fixed $n$, $g(n,k) - \fr(n,k)$ should increase with $k$, the number of colors.  This leads into the third problem.

\begin{problem}
Determine asymptotic bounds, both upper and lower, on the sizes of $\fr(n,k)$ and/or $g(n,k)$.
\end{problem}

For all values of $k$, there are analogous problems that arise when the regularity condition is dropped.  For $k < 6$, this means that some colors might not appear or certain (or any) cubes.  For $k \geq 6$, this means that a color may appear more than once on a cube.

Finally, a number of people have analyzed the complexity of puzzles like the Colored Cubes Puzzle.   A well-known example is  Instant Insanity$\textsuperscript{\textregistered}$, a $4$-colored puzzle whose elegant graph theoretic solution is presented in many introductory texts on combinatorics (see \cite{Chartrand}, for example).  Robertson and Munro showed in \cite{RobertsonMunro} that the solution to a generalization of the Instant Insanity Puzzle with $n$ cubes and $n$ colors is NP-complete.  A more recent work by Demaine, et. al. \cite{Demaine} studied variations of Instant Insanity with several types of prisms; some of these puzzles have solutions that are NP-complete, others can be solved in polynomial time.  There is an analogous problem for the Colored Cubes Puzzle.

\begin{problem}
Determine the complexity of solving the $\cubed{n}$ problem with $12n - 16$ $n$-colored cubes.
\end{problem}

Towards the end of the writing of this paper, we became aware of an arXiv preprint ``On a Generalization of the Eight Blocks to Madness Puzzle''  by Kazuya Haraguchi \cite{Haraguchi}, where, among other things, Haraguchi determines the value of $\fr(2)$.  Section \ref{sec:2frame} of this work and the preprint share a number of results, like Lemma \ref{lem:10cases} and the example at the beginning of Section \ref{sec:2frame} of a collection of $23$ cubes without a corner solution.   The proofs in \cite{Haraguchi} differ from ours, and rely on encoding a subset of cubes into an associated multigraph.  The existence of corner solutions can then be determined by counting the number of tree components in the graph.  The arguments are clever, and we recommend the article to the interested reader.  Some arguments in \cite{Haraguchi} also utilize the tableau, although to a lesser extent than in this paper.

\vfill

\eject
\appendix{Appendix: The Cube Tableau}

\medskip

\begin{tikzpicture}[scale=0.55]
\draw [red] (-5.5,-1) rectangle (-1.5,4.5);
\cubepic{6}{1}{2}{5}{4}{3}{0}{0}
\cubepic{2}{6}{5}{3}{4}{1}{4}{0}
\cubepic{6}{2}{1}{3}{4}{5}{8}{0}
\cubepic{1}{6}{3}{5}{4}{2}{12}{0}
\cubepic{3}{1}{5}{2}{4}{6}{16}{0}
\draw [red] (-1.5,-7) rectangle (2.5,-1.5);
\cubepic{6}{5}{2}{1}{4}{3}{-4}{-6}
\cubepic{1}{5}{6}{3}{4}{2}{4}{-6}
\cubepic{5}{1}{2}{3}{4}{6}{8}{-6}
\cubepic{2}{6}{1}{3}{4}{5}{12}{-6}
\cubepic{2}{5}{3}{6}{4}{1}{16}{-6}
\cubepic{3}{6}{5}{1}{4}{2}{0}{-12}
\draw [red] (2.5,-13) rectangle (6.5,-7.5);
\cubepic{2}{3}{5}{6}{4}{1}{-4}{-12}
\cubepic{2}{6}{1}{5}{4}{3}{8}{-12}
\cubepic{2}{5}{3}{1}{4}{6}{12}{-12}
\cubepic{6}{3}{2}{1}{4}{5}{16}{-12}
\cubepic{2}{1}{5}{3}{4}{6}{0}{-18}
\cubepic{1}{6}{2}{5}{4}{3}{4}{-18}
\draw [red] (6.5,-19) rectangle (10.5,-13.5);
\cubepic{2}{6}{3}{1}{4}{5}{-4}{-18}
\cubepic{3}{6}{5}{2}{4}{1}{12}{-18}
\cubepic{5}{6}{1}{3}{4}{2}{16}{-18}
\cubepic{1}{6}{2}{3}{4}{5}{0}{-24}
\cubepic{5}{2}{1}{3}{4}{6}{4}{-24}
\cubepic{6}{3}{2}{5}{4}{1}{8}{-24}
\draw [red] (10.5,-25) rectangle (14.5,-19.5);
\cubepic{6}{1}{5}{3}{4}{2}{-4}{-24}
\cubepic{6}{5}{1}{2}{4}{3}{16}{-24}
\cubepic{6}{3}{5}{2}{4}{1}{0}{-30}
\cubepic{3}{6}{1}{2}{4}{5}{4}{-30}
\cubepic{6}{5}{3}{1}{4}{2}{8}{-30}
\cubepic{5}{6}{2}{1}{4}{3}{12}{-30}
\draw [red] (14.5,-31) rectangle (18.5,-25.5);
\cubepic{5}{1}{3}{2}{4}{6}{-4}{-30}
\end{tikzpicture}

\nocite{*}

\bibliography{ColoredCubes}
\bibliographystyle{plain}
\end{document}